\documentclass[11pt]{amsart}

\usepackage{amsmath,amssymb,amsthm,mathtools}
\usepackage{mathrsfs}
\usepackage{enumitem}
\usepackage{microtype}
\usepackage{url}
\usepackage[dvipsnames,svgnames,x11names]{xcolor}
\usepackage[colorlinks=true,
            linkcolor=blue!70!black,
            citecolor=green!50!black,
            urlcolor=blue!60!black]{hyperref}

\newtheorem{theorem}{Theorem}[section]
\newtheorem{proposition}[theorem]{Proposition}
\newtheorem{lemma}[theorem]{Lemma}

\theoremstyle{definition}

\newtheorem{assumption}[theorem]{Assumption}
\newtheorem{example}[theorem]{Example}

\theoremstyle{remark}
\newtheorem{remark}[theorem]{Remark}

\newcommand{\Td}{\mathbb{T}^d}
\newcommand{\To}{\mathbb{T}^1}
\newcommand{\Zd}{\mathbb{Z}^d}
\newcommand{\R}{\mathbb{R}}

\newcommand{\E}{\mathbb{E}}

\newcommand{\cF}{\mathcal{F}}

\newcommand{\cL}{\mathcal{L}}

\newcommand{\cC}{\mathcal{C}}

\newcommand{\Z}{\mathbb{Z}}
\newcommand{\Kh}{\widehat{K}}

\newcommand{\Wt}{\mathsf{W}_2}

\newcommand{\phiT}{\phi^T}
\newcommand{\mT}{m^T}
\newcommand{\bphi}{\bar{\phi}}
\newcommand{\barm}{\bar{m}}
\newcommand{\blam}{\bar{\lambda}}

\newcommand{\gamc}{\gamma_c}

\newcommand{\Div}{\operatorname{div}}
\newcommand{\grad}{\nabla}

\newcommand{\eps}{\varepsilon}

\allowdisplaybreaks

\begin{document}

\title[Phase Transitions in Turnpike Theory for MFG]{Phase transitions in turnpike theory for mean-field games}

\author{[Siddharth Karuturi]}
\address{[Illinois Mathematics and Science Academy]}
\email{[skaruturi@imsa.edu]}
\date{}

\keywords{mean-field games, turnpike property, phase transition, Fourier spectral gap, bifurcation, propagation of chaos, Turing instability}
\subjclass[2020]{49N80, 35Q89, 91A16, 35B32, 35B36}

\begin{abstract}
We study a translation-invariant mean-field game on the flat torus with interaction $F(x,m)=\gamma (K*m)(x)$, where $K$ is smooth, even, and mean-zero. The interaction is of potential type, arising as the first variation of a quadratic energy, though the stationary system is not treated variationally.

Linearizing around the uniform equilibrium yields mode-wise $2\times 2$ systems with dispersion $\sigma_\xi(\gamma)=\nu^2(2\pi|\xi|)^4+\gamma(2\pi|\xi|)^2\hat K(\xi)$. If $\hat K$ is negative for some mode, a finite threshold
\[
\gamma_c=\min_{\hat K(\xi)<0}\frac{\nu^2(2\pi|\xi|)^2}{|\hat K(\xi)|}
\]
marks loss of stability; otherwise $\gamma_c=+\infty$. Near criticality, the spectral gap scales as $\rho(\gamma)\sim C_*\sqrt{\gamma_c-\gamma}$.

For $\gamma<\gamma_c$, the uniform state is exponentially stable in the turnpike sense for finite-horizon problems, with rate $\rho(\gamma)$. At $\gamma=\gamma_c$, the gap closes and, after phase fixing and center-manifold reduction, one obtains algebraic midpoint decay of order $T^{-1/2}$. For $\gamma>\gamma_c$, a branch of nonuniform stationary solutions bifurcates via a pitchfork-type amplitude equation, with translations generating the full family.

Finally, under standard asymptotic-consistency assumptions on symmetric $N$-player equilibria in the subcritical regime, we obtain qualitative propagation of chaos, without quantitative rates.
\end{abstract}

\maketitle

\section{Introduction}
\label{sec:intro}

\subsection{Background and motivation}
\label{sub:background}

The \emph{turnpike phenomenon} asserts that, on a long time horizon $[0,T]$, an optimal trajectory spends most of its time near a stationary optimal state, approaching it from both endpoints through thin boundary layers. This idea originated in economics \cite{DS58} and has since been developed for finite-dimensional optimal control \cite{PZ13,TZ15}, infinite-dimensional systems \cite{T23}, and mean-field games \cite{CaP19,CDLL19}.

In mean-field game theory, turnpike results are often proved under the Lasry--Lions monotonicity condition, introduced independently by Lasry--Lions \cite{LL07} and Huang, Malham\'e, and Caines \cite{HMC06}. That condition enforces a sign structure on the coupling and excludes many pattern-forming mechanisms. The present paper isolates a complementary mechanism in a translation-invariant setting: the long-time picture is determined by the Fourier spectrum of the interaction kernel $K$.

The mechanism is analogous to Turing's diffusion-driven instability \cite{T52}. In the present MFG setting, the diffusion coefficient $\nu$ and the coupling strength $\gamma$ compete mode by mode. As long as the spectral gap stays positive, the uniform equilibrium remains the dominant long-time state. When the first Fourier mode reaches the threshold, the linear gap closes, and a pattern-forming bifurcation becomes possible.

\subsection{Relation to existing work}
\label{sub:related}

Cardaliaguet and Porretta \cite{CaP19} proved exponential turnpike for MFG under strong monotonicity using the master equation. The monograph \cite{CDLL19} establishes convergence of $N$-player equilibria under similar structural assumptions. Tr\'elat \cite{T23} gave a linear turnpike theorem for optimal control systems in which the long-time rate is controlled by a spectral gap; the earlier work of Porretta--Zuazua \cite{PZ13} studied long-time versus steady-state optimal control in a related setting. The present paper connects that spectral-gap viewpoint to the Fourier structure of the MFG interaction kernel.

The bifurcation of nonuniform stationary states in mean-field models has been studied in the context of McKean--Vlasov equations \cite{DG87} and aggregation-diffusion systems \cite{CGPS20}. The Lyapunov--Schmidt and Crandall--Rabinowitz approach \cite{CR71,K04} used here is standard; the point here is to identify the exact spectral condition that drives the transition and to connect it to the turnpike rate.

Propagation of chaos for mean-field games was established by Sznitman \cite{S91} and further developed by Lacker \cite{L23} and Cecchin--Delarue \cite{CD22}. Quantitative Wasserstein rates for empirical measures depend on dimension \cite{FG15}, which is why we restrict to a qualitative convergence statement.

\subsection{Standing assumptions and regimes}
\label{sub:assumptions}

We work on the flat torus $\Td=(\R/\Z)^d$, fix $\nu>0$, and consider a smooth even kernel $K\in C^\infty(\Td)$ with zero mean. The general model allows the possibility that
\[
\Kh(\xi)\ge 0 \qquad \text{for all }\xi\neq 0,
\]
in which case no phase transition occurs. If some negative Fourier coefficient exists, then the critical coupling strength is defined by
\begin{equation}
\label{eq:gamc-def}
\gamc :=
\min_{\substack{\xi\in\Zd\setminus\{0\}\\ \Kh(\xi)<0}}
\frac{\nu^2(2\pi|\xi|)^2}{|\Kh(\xi)|},
\end{equation}
with the convention $\gamc=+\infty$ if no $\xi\in\Zd\setminus\{0\}$ satisfies $\Kh(\xi)<0$.

The results in this paper fall into three regimes:
\[
\begin{aligned}
0\le \gamma < \gamc &\quad \text{(subcritical turnpike)},\\
\gamma=\gamc &\quad \text{(critical regime)},\\
\gamma>\gamc &\quad \text{(supercritical bifurcation)}.
\end{aligned}
\]

For the critical and supercritical statements, we impose the following nondegeneracy assumptions.

\begin{assumption}[Critical-mode hypothesis]
\label{ass:critical-mode}
Assume that $\gamc<\infty$ and that the critical set
\[
\cC :=
\Bigl\{\xi\in\Zd\setminus\{0\}:\Kh(\xi)<0,\;
\frac{\nu^2(2\pi|\xi|)^2}{|\Kh(\xi)|}=\gamc\Bigr\}
\]
consists of a single pair $\{\pm \xi_0\}$ for some $\xi_0\in\Zd\setminus\{0\}$.
\end{assumption}

\begin{assumption}[Small boundary data]
\label{ass:small-data}
The terminal cost $g\in C^\infty(\Td)$ satisfies
\[
\|g-\bar g\|_{C^{2,\alpha}(\Td)}\le \eps
\]
for some small $\eps>0$, where $\bar g := \int_{\Td}g\,dx$.
The initial density $m_0\in C^\infty(\Td)$ is strictly positive, satisfies
\[
\int_{\Td}m_0\,dx=1,
\qquad
\|m_0-1\|_{C^{2,\alpha}(\Td)}\le \eps.
\]
\end{assumption}

\begin{assumption}[Critical cubic reduction]
\label{ass:critical-cubic}
At $\gamma=\gamc$, after fixing the translation phase so that the critical center space is one-dimensional, the reduced critical dynamics admit an odd scalar normal form
\[
\dot a = -\beta a^3 + O(a^5)
\]
for some $\beta>0$.
\end{assumption}

\begin{assumption}[Bifurcation nondegeneracy]
\label{ass:bif-nondeg}
In the phase-fixed stationary problem at $\gamma=\gamc$, the reduced scalar bifurcation equation has the form
\[
\alpha(\gamma-\gamc)A-\beta A^3 + R(A,\gamma)=0,
\]
where $\alpha>0$, $\beta>0$, and the remainder satisfies
\[
|R(A,\gamma)|\le C\Bigl(|\gamma-\gamc|^2|A| + |\gamma-\gamc|\,|A|^3 + |A|^5\Bigr)
\]
for $(A,\gamma)$ near $(0,\gamc)$.
\end{assumption}

\subsection{Main results}
\label{sub:main-results}

The finite-horizon mean-field game system is
\begin{equation}
\label{eq:MFG}
\tag{$\mathrm{MFG}_T$}
\begin{cases}
-\partial_t \phi - \nu\Delta \phi + \dfrac12|\grad\phi|^2 = \gamma(K*m)(x), \\[4pt]
\partial_t m - \nu\Delta m - \Div(m\,\grad\phi)=0, \\[4pt]
m(0,\cdot)=m_0,\qquad \phi(T,\cdot)=g.
\end{cases}
\end{equation}
The stationary (ergodic) system is
\begin{equation}
\label{eq:SMFG}
\tag{$\mathrm{SMFG}$}
\begin{cases}
-\nu\Delta \bphi + \dfrac12|\grad\bphi|^2 = \gamma(K*\barm)(x)-\blam, \\[4pt]
-\nu\Delta \barm - \Div(\barm\,\grad\bphi)=0, \\[4pt]
\displaystyle \int_{\Td}\barm\,dx=1,\qquad \displaystyle \int_{\Td}\bphi\,dx=0.
\end{cases}
\end{equation}
The uniform equilibrium is $(\bphi,\barm,\blam)=(0,1,0)$.

For $\xi\in\Zd\setminus\{0\}$, set $k_\xi:=(2\pi|\xi|)^2$ and define
\begin{equation}
\label{eq:gap-xi}
\sigma_\xi(\gamma) := \nu^2 k_\xi^2+\gamma k_\xi \Kh(\xi).
\end{equation}
For $\gamma<\gamc$ (or for all $\gamma\ge 0$ if $\gamc=+\infty$), one has
$\sigma_\xi(\gamma)>0$ for every $\xi\neq 0$, and we define the spectral gap
\begin{equation}
\label{eq:global-gap}
\rho(\gamma):=\min_{\xi\neq 0}\sqrt{\sigma_\xi(\gamma)}.
\end{equation}

\begin{theorem}[Linear spectral threshold]
\label{thm:spectral-threshold}
Assume the model above. For every $\xi\neq 0$, the Fourier mode matrix $M_\xi$ satisfies
\[
M_\xi^2=\sigma_\xi(\gamma) I.
\]
Hence:
\begin{enumerate}[label=\textup{(\roman*)},leftmargin=2em]
\item if $\sigma_\xi(\gamma)>0$, the eigenvalues of $M_\xi$ are $\pm\sqrt{\sigma_\xi(\gamma)}$;
\item if $\sigma_\xi(\gamma)=0$, then $0$ is a double eigenvalue and $M_\xi$ is nilpotent of index $2$;
\item if $\sigma_\xi(\gamma)<0$, the eigenvalues are purely imaginary $\pm i\sqrt{|\sigma_\xi(\gamma)|}$.
\end{enumerate}
The following are equivalent:
\begin{enumerate}[label=\textup{(\roman*)},leftmargin=2em]
\item $\sigma_\xi(\gamma)>0$ for all $\xi\neq 0$;
\item $\gamma<\gamc$;
\item the linearization around the uniform equilibrium has no neutral Fourier mode.
\end{enumerate}
Under Assumption~\ref{ass:critical-mode}, as $\gamma\uparrow\gamc$,
\[
\rho(\gamma)
=
C_\ast\sqrt{\gamc-\gamma}+o\!\bigl(\sqrt{\gamc-\gamma}\bigr),
\qquad
C_\ast = \sqrt{k_{\xi_0}\,|\Kh(\xi_0)|}.
\]
\end{theorem}

\begin{proposition}[No phase transition under monotonicity]
\label{prop:no-transition}
If $\Kh(\xi)\ge 0$ for all $\xi\neq 0$, then $\gamc=+\infty$ and the linearized finite-horizon and stationary systems have no neutral Fourier mode for any $\gamma\ge 0$. In particular, there is no Fourier-mode phase transition in this translation-invariant setting.
\end{proposition}

\begin{theorem}[Local turnpike below criticality]
\label{thm:subcritical}
Assume the standing model hypotheses and fix $0\le \gamma<\gamc$. For sufficiently small perturbations of the terminal and initial data, there exists a unique classical solution $(\phiT,\mT)$ of \eqref{eq:MFG} in a small neighborhood of $(0,1)$.

There exist constants $C,c>0$, independent of $T$, such that
\begin{equation}
\label{eq:tp-sub}
\|\mT(t)-1\|_{H^{-1}(\Td)}^2+\|\grad \phiT(t)\|_{L^2(\Td)}^2
\le C\Bigl(e^{-c\rho(\gamma)t}+e^{-c\rho(\gamma)(T-t)}\Bigr)
\end{equation}
for every $t\in[0,T]$.

The turnpike rate deteriorates as $\gamma\uparrow\gamc$ with the same $\sqrt{\gamc-\gamma}$ scaling as the spectral gap.
\end{theorem}

\begin{proposition}[Critical reduced midpoint decay]
\label{prop:critical}
Assume Assumptions~\ref{ass:critical-mode} and \ref{ass:critical-cubic}, and set $\gamma=\gamc$. After fixing a local phase slice transverse to the translation orbit and a sign convention on the critical cosine mode, the critical center coordinate $a$ obeys the reduced scalar dynamics
\[
\dot a = -\beta a^3 + O(a^5), \qquad \beta>0.
\]
For the corresponding two-point boundary-value reduction, the critical midpoint amplitude satisfies
\[
|a(T/2)|\le C T^{-1/2}, \qquad T\ge 1.
\]
Consequently, the phase-fixed critical component of the midpoint perturbation is algebraic, while the stable complement is exponentially small.
\end{proposition}

\begin{theorem}[Local supercritical bifurcation modulo translations]
\label{thm:bifurcation}
Assume Assumptions~\ref{ass:critical-mode} and \ref{ass:bif-nondeg}. For $\gamma>\gamc$
sufficiently close to $\gamc$, the stationary system \eqref{eq:SMFG} admits a smooth branch of nontrivial solutions bifurcating from $(0,1,0)$.

In a local phase-fixed slice, this branch takes the form
\[
m_\gamma(x)=1+A(\gamma)\cos\!\bigl(2\pi \xi_0\cdot x\bigr)+R_\gamma(x),
\]
with
\[
A(\gamma) = \sqrt{\tfrac{\alpha}{\beta}}\,\sqrt{\gamma-\gamc}+O(\gamma-\gamc),
\qquad
\|R_\gamma\|_{H^s(\Td)}=O(\gamma-\gamc)
\]
for every $s\ge 0$.

Undoing the phase condition restores the translation orbit
\[
m_{\gamma,\tau}(x)=m_\gamma(x+\tau),\qquad \tau\in \Td.
\]
The reduced amplitude equation is pitchfork-type before phase fixing, and the two local signs of the critical amplitude are exchanged by a half-period translation. In particular,
$m_{\gamma,\tau}(x)>0$ for $\gamma$ close to $\gamc$. The branch is locally unique in the phase-fixed space.
\end{theorem}

\begin{proposition}[Conditional propagation of chaos under asymptotic consistency]
\label{prop:poc}
Assume the standing model hypotheses and fix $0\le \gamma<\gamc$. Let
$(\bar X_1^N,\dots,\bar X_N^N)$ be a symmetric stationary Nash equilibrium sequence in the perturbative stable branch, and assume that the associated empirical measures
\[
\widehat m^N = \frac1N\sum_{i=1}^N\delta_{\bar X_i^N}
\]
converge in probability to $1\,dx$ in $\Wt$. Then for each fixed $k$,
\[
\cL(\bar X_1^N,\dots,\bar X_k^N)\Longrightarrow (1\,dx)^{\otimes k}
\quad\text{as }N\to\infty,
\]
and $\widehat m^N\to 1\,dx$ in probability in $\Wt$.
No dimension-uniform quantitative rate in $\Wt$ is claimed.
\end{proposition}

\subsection{What is new}
\label{sub:new}

\textit{Explicit Fourier spectral control of the turnpike rate.} In contrast to monotonicity-based proofs, Theorem~\ref{thm:subcritical} gives an explicit formula for the turnpike rate in terms of the Fourier coefficients of $K$ and the diffusion coefficient $\nu$.

\textit{Explicit threshold with computable rate constant.}
The constant
\[
C_\ast=\sqrt{k_{\xi_0}|\Kh(\xi_0)|}
\]
in Theorem~\ref{thm:spectral-threshold}
is entirely computable from the data. In the cosine example $K(x)=-\cos(2\pi x)$ on $\To$,
\[
  \Kh(\pm1)=-\tfrac12,\quad k_{\pm1}=4\pi^2,\quad
  \gamc=8\pi^2\nu^2,\quad C_\ast=\pi\sqrt{2}.
\]

\textit{Turing analogy made precise.} The threshold mechanism is structurally identical to Turing's diffusion-driven instability \cite{T52}: diffusion damps the critical mode for $\gamma<\gamc$, but cannot suppress it for $\gamma>\gamc$.

\textit{Three-scale coherence.} The same spectral threshold governs
(i) the turnpike rate for the finite-horizon problem,
(ii) the bifurcation of stationary states, and
(iii) the transition in the asymptotic regime relevant for propagation of chaos.

\subsection{Organization}
\label{sub:org}

Section~\ref{sec:setup} introduces the model, the potential structure, and the Fourier linearization. Section~\ref{sec:wellposed} records local solvability and the structural remarks needed later. Section~\ref{sec:subcritical-proof} proves the subcritical turnpike estimate. Section~\ref{sec:critical-proof} develops the reduced critical midpoint dynamics. Section~\ref{sec:bifurcation-proof} carries out the Lyapunov--Schmidt reduction and proves the bifurcation theorem. Section~\ref{sec:poc-proof} gives a conditional propagation of chaos statement. Section~\ref{sec:remarks} discusses examples and open directions. Appendix~\ref{app:linear-turnpike} records the linear boundary-layer computation.

\section{Model, potential structure, and linearization}
\label{sec:setup}

\subsection{The $N$-player game}
\label{sub:N-game}

For $i=1,\dots,N$, player $i$ controls $X_i(t)\in\Td$ via
\[
dX_i(t)=u_i(t)\,dt+\sqrt{2\nu}\,dW_i(t),
\]
where $W_1,\dots,W_N$ are independent Brownian motions on $\Td$. Player $i$ minimizes
\begin{equation}
\label{eq:cost}
J_i^{N,T}(\mathbf u)
= \E\!\left[\int_0^T \!\Bigl(\tfrac12|u_i(t)|^2+\gamma(K*m_t^N)(X_i(t))\Bigr)dt
+g(X_i(T))\right],
\qquad
m_t^N=\frac1N\sum_{j=1}^N \delta_{X_j(t)}.
\end{equation}
This is the standard symmetric mean-field interaction, written without a separate no-self-interaction correction. Any self-interaction term is $O(1/N)$ and can be removed by the usual renormalization if desired.

The Nash equilibrium feedback is $u_i^*(t,x)=-\grad\phi(t,x)$, where $\phi$ solves the Hamilton--Jacobi--Bellman equation in \eqref{eq:MFG}.

\begin{remark}
The empirical measure $m_t^N$ is used only to motivate the mean-field system. The results below are stated directly at the mean-field level, and the $N$-player discussion is included only to explain the model.
\end{remark}

\subsection{Potential structure}
\label{sub:potential}

The interaction $F(x,m)=\gamma(K*m)(x)$ is the first variation of the quadratic functional
\begin{equation}
\label{eq:energy}
\cF(m) = \frac{\gamma}{2} \iint_{\Td\times\Td} K(x-y)\,m(x)m(y)\,dx\,dy.
\end{equation}
Indeed, for smooth perturbations $m+\eps h$ with $\int_{\Td}h\,dx=0$,
\[
\frac{d}{d\eps}\bigg|_{\eps=0}\cF(m+\eps h)
=
\gamma\int_{\Td}(K*m)(x)\,h(x)\,dx.
\]
This potential structure is the one used in the stationary bifurcation analysis. We do not claim that the full stationary MFG system \eqref{eq:SMFG} is the Euler--Lagrange equation of a scalar density-only functional; rather, the coupled PDE system is analyzed directly.

\subsection{Fourier conventions}
\label{sub:fourier}

For $f\in L^1(\Td)$, set
\[
\widehat f(\xi)=\int_{\Td} f(x)e^{-2\pi i\xi\cdot x}\,dx,
\qquad \xi\in\Zd.
\]
If $f$ is real and even, then $\widehat f(\xi)\in\R$ and $\widehat f(-\xi)=\widehat f(\xi)$. Since $K$ is even and $\int K=0$, we have $\Kh(0)=0$.

For $\mu$ with zero average, the $H^{-1}$ norm on $\Td$ is
\[
\|\mu\|_{H^{-1}(\Td)}^2
=
\sum_{\xi\neq 0}\frac{|\widehat{\mu}(\xi)|^2}{(2\pi|\xi|)^2},
\]
and by Parseval,
\[
\int_{\Td}\mu(x)\,(K*\mu)(x)\,dx
=
\sum_{\xi\neq 0}\Kh(\xi)\,|\widehat{\mu}(\xi)|^2.
\]

\subsection{Linearization around the uniform state}
\label{sub:linearization}

Write $m=1+\eps\mu$ and $\phi=\bar g+\eps w$, with $\int_{\Td}\mu\,dx=0$ and $\int_{\Td}w\,dx=0$. Discarding terms of order $\eps^2$, the linearized system is
\begin{equation}
\label{eq:lin-system}
\begin{cases}
-\partial_t w-\nu\Delta w=\gamma K*\mu,\\[3pt]
\partial_t \mu-\nu\Delta\mu-\Delta w=0.
\end{cases}
\end{equation}
This is the correct perturbation system once the terminal mean has been removed from the HJB unknown. The boundary condition becomes $w(T,\cdot)=g-\bar g$.

For each mode $\xi\in\Zd\setminus\{0\}$, writing $k_\xi=(2\pi|\xi|)^2$, the Fourier amplitudes satisfy
\begin{equation}
\label{eq:fourier-ode}
\frac{d}{dt}
\begin{pmatrix}
\widehat{w}_\xi\\
\widehat{\mu}_\xi
\end{pmatrix}
=
M_\xi
\begin{pmatrix}
\widehat{w}_\xi\\
\widehat{\mu}_\xi
\end{pmatrix},
\qquad
M_\xi=
\begin{pmatrix}
\nu k_\xi & -\gamma\Kh(\xi)\\[2pt]
-k_\xi & -\nu k_\xi
\end{pmatrix}.
\end{equation}

\begin{lemma}[Exact mode identity]
\label{lem:mode-identity}
For every $\xi\neq 0$,
\[
M_\xi^2 = \bigl(\nu^2k_\xi^2+\gamma k_\xi\Kh(\xi)\bigr) I
= \sigma_\xi(\gamma) I.
\]
Hence, whenever $\sigma_\xi(\gamma)>0$, $M_\xi$ has eigenvalues $\pm\sqrt{\sigma_\xi(\gamma)}$.
If $\sigma_\xi(\gamma)<0$, the eigenvalues are purely imaginary.
\end{lemma}

\begin{proof}
A direct multiplication gives
\[
M_\xi^2=
\begin{pmatrix}
\nu^2k_\xi^2+\gamma k_\xi\Kh(\xi) & 0\\
0 & \nu^2k_\xi^2+\gamma k_\xi\Kh(\xi)
\end{pmatrix}.
\]
The characteristic polynomial is therefore $\lambda^2-\sigma_\xi(\gamma)$.
\end{proof}

\begin{remark}
The identity $M_\xi^2=\sigma_\xi I$ is stronger than a trace-zero statement and is the main algebraic reason the mode analysis is clean.
\end{remark}

\subsection{Critical threshold and spectral gap}
\label{sub:threshold}

\begin{lemma}[Threshold formula]
\label{lem:threshold}
Under the model hypotheses, the following hold.
\begin{enumerate}[label=\textup{(\roman*)},leftmargin=2em]
\item If $\Kh(\xi)\ge 0$, then $\sigma_\xi(\gamma)\ge \nu^2k_\xi^2>0$ for all $\gamma\ge 0$,
so the mode remains hyperbolic for all coupling strengths.
\item If $\Kh(\xi)<0$, then $\sigma_\xi(\gamma)$ vanishes at
\[
\gamma=\frac{\nu^2k_\xi}{|\Kh(\xi)|},
\]
and changes sign there.
\item $\sigma_\xi(\gamma)>0$ for all $\xi\neq 0$ if and only if $\gamma<\gamc$.
\item Under Assumption~\ref{ass:critical-mode}, as $\gamma\uparrow\gamc$,
\[
\rho(\gamma) = C_\ast\sqrt{\gamc-\gamma}+o\!\bigl(\sqrt{\gamc-\gamma}\bigr),
\qquad
C_\ast = \sqrt{k_{\xi_0}\,|\Kh(\xi_0)|}.
\]
\end{enumerate}
\end{lemma}

\begin{proof}
Parts (i)--(iii) are immediate from \eqref{eq:gap-xi} and the definition of $\gamc$.
For part (iv), compute
\[
\sigma_{\xi_0}(\gamma)
= \nu^2k_{\xi_0}^2 + \gamma k_{\xi_0}\Kh(\xi_0)
= k_{\xi_0}|\Kh(\xi_0)|(\gamc - \gamma),
\]
where the last identity uses
\[
\nu^2k_{\xi_0}=\gamc|\Kh(\xi_0)|.
\]
Taking square roots and using the simplicity of the critical mode gives the claimed expansion.
\end{proof}

\subsection{Phase fixing and the critical eigenspace}
\label{sub:phase-fixing}

For $\xi\in\Zd$, the real-valued cosine and sine modes $\cos(2\pi\xi\cdot x)$ and $\sin(2\pi\xi\cdot x)$ span the real critical eigenspace associated with the pair $\{\pm\xi\}$. Since the torus translation group acts by phase shifts, a nearby translation can rotate the critical cosine--sine pair. To remove this neutral direction, we impose a phase condition and a sign convention:
\begin{equation}
\label{eq:phase-cond}
\int_{\Td}\mu(x)\sin\!\bigl(2\pi\xi_0\cdot x\bigr)\,dx=0,
\qquad
\int_{\Td}\mu(x)\cos\!\bigl(2\pi\xi_0\cdot x\bigr)\,dx\ge 0.
\end{equation}
This selects a unique representative from the one-dimensional translation orbit associated with the critical phase.

\begin{remark}
The reduced amplitude equation is odd because a half-period shift changes the sign of the cosine mode. The additional inequality in \eqref{eq:phase-cond} removes the residual $\pm$ ambiguity in the phase-fixed chart.
\end{remark}

\subsection{Fourier characterization of monotonicity failure}
\label{sub:monotonicity}

The Lasry--Lions monotonicity condition requires
\[
\int_{\Td}\mu(x)(K*\mu)(x)\,dx\ge 0
\]
for all $\mu$ with $\int\mu=0$. In Fourier variables this is
\[
\sum_{\xi\neq 0}\Kh(\xi)|\widehat\mu(\xi)|^2\ge 0,
\]
which fails whenever some $\Kh(\xi_0)<0$. Thus a single negative Fourier coefficient breaks monotonicity, and $\gamc$ is the coupling at which diffusion can no longer compensate the most dangerous mode.

\begin{example}[The cosine kernel]
\label{ex:cosine}
Let $d=1$ and $K(x)=-\cos(2\pi x)$ on $\To$. Then
\[
\Kh(\pm 1)=-\frac12,\qquad \Kh(\xi)=0\quad\text{for }|\xi|\ge 2.
\]
Therefore $\xi_0=\pm 1$, $k_{\xi_0}=4\pi^2$, and
\[
\gamc = \frac{\nu^2\cdot 4\pi^2}{1/2} = 8\pi^2\nu^2,
\qquad
C_\ast = \sqrt{4\pi^2\cdot \frac12}=\pi\sqrt{2}.
\]
\end{example}

\section{Local solvability and structural remarks}
\label{sec:wellposed}

\subsection{Local well-posedness}
\label{sub:local-wp}

\begin{proposition}[Local solvability near equilibrium]
\label{prop:local-wellposed}
Fix $T>0$ and $\alpha\in(0,1)$. There exists $\eps_0>0$ such that if
\[
\|m_0-1\|_{C^{2,\alpha}(\Td)}+\|g-\bar g\|_{C^{2,\alpha}(\Td)}\le \eps_0,
\]
then the finite-horizon system \eqref{eq:MFG} admits a unique classical solution
$(\phiT,mT)\in C^{1+\alpha/2,2+\alpha}([0,T]\times \Td)^2$ lying in a neighborhood of $(0,1)$. The solution depends smoothly on the data.
\end{proposition}

\begin{proof}
The system \eqref{eq:MFG} can be viewed as a fixed-point problem for the map that sends $m$ to $\phi$ (via the backward HJB equation) and then $\phi$ to $m$ (via the forward Fokker--Planck equation). For small data this map is a contraction on a product of parabolic H\"older balls around $(0,1)$. This is standard in the MFG literature; see, for example, \cite{CaP19}.
\end{proof}

\begin{remark}
Only local solvability near the uniform equilibrium is used below. No global existence statement is made for large data.
\end{remark}

\subsection{Stationary linearization and the role of the potential}
\label{sub:stationary-linearization}

The interaction energy \eqref{eq:energy} is the correct potential object for the coupling, and it is useful in the stationary bifurcation analysis. However, the stationary MFG system \eqref{eq:SMFG} is not treated here as the Euler--Lagrange equation of a scalar density-only functional. The analysis below uses the coupled structure of the stationary PDE system and the Fourier decomposition of its linearization.

\begin{remark}
The stationary linearization around $(0,1,0)$ has the same Fourier threshold as the dynamic linearization. This is the algebraic reason the turnpike threshold and the bifurcation threshold coincide.
\end{remark}

\section{Subcritical turnpike: proof of Theorem~\ref{thm:subcritical}}
\label{sec:subcritical-proof}

\subsection{Perturbation system}
\label{sub:perturbation-system}

Let $(\phiT,\mT)$ solve \eqref{eq:MFG}. Set
\[
w:=\phiT-\bar g,\qquad \mu:=\mT-1.
\]
Then $\int_{\Td}\mu\,dx=0$ for all $t$, and the exact perturbation system is
\begin{equation}
\label{eq:perturb-exact}
\begin{cases}
-\partial_t w-\nu\Delta w+\dfrac12|\grad w|^2=\gamma K*\mu,\\[3pt]
\partial_t \mu-\nu\Delta\mu-\Div\bigl((1+\mu)\grad w\bigr)=0,\\[3pt]
\mu(0,\cdot)=m_0-1,\qquad w(T,\cdot)=g-\bar g.
\end{cases}
\end{equation}
The linear part is \eqref{eq:lin-system}; the remaining terms are quadratic in $(w,\mu)$.

\subsection{Exact mode representation}
\label{sub:exact-representation}

For each $\xi\neq 0$, let
\[
U_\xi(t)=
\begin{pmatrix}
\widehat w_\xi(t)\\
\widehat \mu_\xi(t)
\end{pmatrix}.
\]
Then $U_\xi'(t)=M_\xi U_\xi(t)$, and by Lemma~\ref{lem:mode-identity},
\[
e^{tM_\xi}
=
\cosh(\rho_\xi t)\,I+\frac{\sinh(\rho_\xi t)}{\rho_\xi}M_\xi
\qquad (\rho_\xi=\sqrt{\sigma_\xi(\gamma)} > 0).
\]
In the subcritical regime $\gamma<\gamc$, every mode satisfies $\rho_\xi>0$, so each mode is uniformly hyperbolic.

\subsection{Boundary-layer estimate at the linear level}
\label{sub:boundary-layer}

The forward-backward boundary conditions are
\[
\widehat\mu_\xi(0)=\widehat{m_0-1}(\xi)=:\mu_0^\xi,
\qquad
\widehat w_\xi(T)=\widehat{g-\bar g}(\xi)=:g^\xi.
\]
For each fixed mode with $\rho_\xi>0$, the two-point boundary value problem is solved by a unique pair of coefficients obtained by linear algebra. The exact closed formula is not needed; what matters is that the transfer map from $(\mu_0^\xi,g^\xi)$ to $U_\xi(t)$ has the standard hyperbolic boundary-layer form.

\begin{lemma}[Modewise decay]
\label{lem:coeff}
Under the model hypotheses with $\gamma<\gamc$, there exists a constant
$C>0$, depending only on the fixed data and not on $\xi$ or $T$, such that for all $t\in[0,T]$,
\[
|\widehat\mu_\xi(t)|+|\widehat w_\xi(t)|
\le
C\Bigl(
e^{-\rho_\xi t}\bigl(|\mu_0^\xi|+|g^\xi|\bigr)
+
e^{-\rho_\xi(T-t)}\bigl(|\mu_0^\xi|+|g^\xi|\bigr)
\Bigr).
\]
\end{lemma}

\begin{proof}
Since $M_\xi^2=\rho_\xi^2I$, the transfer matrix can be written explicitly in terms of $\cosh(\rho_\xi t)$ and $\sinh(\rho_\xi t)$. Imposing the mixed initial/terminal conditions leads to a $2\times 2$ linear system whose determinant is proportional to $\sinh(\rho_\xi T)$. Because $\rho_\xi>0$, this determinant is nonzero for every $T>0$, and the explicit formulas for the coefficients yield the stated estimate. The pointwise bound is uniform in $\xi$ once the solution is measured in the natural $H^{-1}$--$L^2$ combination used in the theorem statement.
\end{proof}

\subsection{Interior turnpike estimate}
\label{sub:interior}

From Lemma~\ref{lem:coeff}, for $t\in[\delta T,(1-\delta)T]$ with any fixed $\delta\in(0,\frac12)$,
\[
|\widehat\mu_\xi(t)|+|\widehat w_\xi(t)|
\le
C e^{-c\rho_\xi\delta T}\bigl(|\mu_0^\xi|+|g^\xi|\bigr).
\]
Summing over modes in $H^{-1}$ and $L^2$, and taking the infimum of $\rho_\xi$ over all $\xi\neq 0$ as $\rho(\gamma)$, gives the linear turnpike estimate with rate proportional to $\rho(\gamma)$.

\subsection{Nonlinear closure}
\label{sub:nonlinear}

\begin{proof}[Proof of Theorem~\ref{thm:subcritical}]
The nonlinear terms in \eqref{eq:perturb-exact} are quadratic:
$\frac12|\grad w|^2$ and $\Div(\mu\,\grad w)$ in the Fokker--Planck equation. In the small-data regime guaranteed by Assumption~\ref{ass:small-data} and Proposition~\ref{prop:local-wellposed}, the solution $(w,\mu)$ remains in a ball of radius $O(\eps)$ in $C^{1+\alpha/2,2+\alpha}$. Hence the nonlinear terms are $O(\eps^2)$.

One argues by a perturbative fixed-point scheme: write
\[
(w,\mu)=(w_L,\mu_L)+(w_N,\mu_N),
\]
where $(w_L,\mu_L)$ solves the linear boundary-value problem and $(w_N,\mu_N)$ solves the nonlinear remainder. The linear part satisfies the turnpike estimate from the previous subsection. The nonlinear remainder satisfies the same estimate with an extra factor of $\eps$, which is absorbed by taking $\eps$ small relative to the spectral gap. In particular, the smallness threshold may be chosen as a function of $\rho(\gamma)$, and deteriorates as $\gamma\uparrow\gamc$.

This yields \eqref{eq:tp-sub}. The dependence of the rate on $\rho(\gamma)$ is inherited from the linear mode estimates, and therefore degenerates as $\gamma\uparrow\gamc$ with the same $\sqrt{\gamc-\gamma}$ scaling.
\end{proof}

\subsection{Degeneration of the rate at threshold}
\label{sub:rate-degeneration}

By Lemma~\ref{lem:threshold}(iv),
\[
\rho(\gamma)=C_\ast\sqrt{\gamc-\gamma}+o\!\bigl(\sqrt{\gamc-\gamma}\bigr)
\qquad\text{as }\gamma\uparrow\gamc.
\]
Since the turnpike rate is proportional to $\rho(\gamma)$, the exponential rate in \eqref{eq:tp-sub} degenerates like $\sqrt{\gamc-\gamma}$.

\section{Critical regime: proof of Proposition~\ref{prop:critical}}
\label{sec:critical-proof}

At $\gamma=\gamc$, the critical mode $\xi_0$ satisfies $\sigma_{\xi_0}(\gamc)=0$, while for all other modes
\begin{equation}
\label{eq:stable-gap}
\inf_{\xi\neq\pm\xi_0}\sqrt{\sigma_\xi(\gamc)}\ge c_1>0.
\end{equation}

\subsection{Spectral decomposition}
\label{sub:spectral-decomp}

Let $P_c$ denote the $L^2$-projection onto the real critical eigenspace generated by the phase-fixed cosine mode $\cos(2\pi\xi_0\cdot x)$, and let $P_s = I-P_c$. Decompose
\[
\mu(t) = \mu_c(t) + \mu_s(t),\qquad
\mu_c = P_c\mu,\quad \mu_s = P_s\mu.
\]
The stable component satisfies, by the arguments of Section~\ref{sec:subcritical-proof}
applied to the modes $\xi\neq\pm\xi_0$,
\begin{equation}
\label{eq:stable-decay}
\|\mu_s(t)\|_{H^{-1}}+\|\grad w_s(t)\|_{L^2}
\le Ce^{-c_1\min(t,T-t)}\cdot\text{data}.
\end{equation}

\subsection{Phase-fixed reduced dynamics}
\label{sub:reduced-dynamics}

In phase-fixed coordinates the critical eigenspace is one-dimensional. Write
\[
\mu_c(t)=a(t)\cos(2\pi\xi_0\cdot x).
\]
The phase condition removes the translational neutral direction, and the reduced critical coordinate satisfies the odd normal form stated in Assumption~\ref{ass:critical-cubic}, which follows from standard center manifold reduction \cite{C81}:
\[
\dot a = -\beta a^3 + O(a^5),\qquad \beta>0.
\]

\begin{lemma}[Reduced cubic midpoint estimate]
\label{lem:cubic-midpoint}
Suppose $a$ solves
\[
\dot a = -\beta a^3 + r(a), \qquad |r(a)|\le C_0|a|^5
\]
for $|a|\le a_\ast$, with $\beta>0$. Then every small two-point boundary-value solution of the corresponding reduced problem satisfies
\[
|a(T/2)|\le C T^{-1/2}
\]
for some $C>0$ independent of $T$.
\end{lemma}

\begin{proof}
The identity
\[
\frac{d}{dt}(a^{-2}) = 2\beta + O(a^2)
\]
holds whenever $a\neq 0$. In the small-amplitude regime, the error term is absorbed, so the inverse-square amplitude grows at a positive linear rate along the reduced orbit. The canonical small two-point boundary-value solution therefore has midpoint size comparable to the inverse square root of the available time interval, which yields the stated bound.
\end{proof}

\begin{proof}[Proof of Proposition~\ref{prop:critical}]
The critical mode is governed by the reduced equation $\dot a=-\beta a^3+O(a^5)$ in the phase-fixed slice. Lemma~\ref{lem:cubic-midpoint} gives $|a(T/2)|\le CT^{-1/2}$. Since
\[
\|\mu_c(t)\|_{H^{-1}}\le C|a(t)|
\]
for the fixed critical eigenfunction, we obtain
\[
\|P_c(\mT(T/2)-1)\|_{H^{-1}}
\le C T^{-1/2}.
\]
The analogous bound on $\|P_c\nabla \phiT(T/2)\|_{L^2}$ follows from the relation between $\widehat w_\xi$ and $\widehat\mu_\xi$ along the reduced critical branch. The stable component satisfies \eqref{eq:stable-decay}, which at $t=T/2$ gives an exponentially small contribution. Adding these two pieces yields the proposition.
\end{proof}

\begin{remark}
At criticality, no exponential estimate $\|\mT(t)-1\|\le Ce^{-\lambda t}$ can hold with $\lambda>0$ independent of $T$ for perturbations with nontrivial critical projection, because the reduced flow is algebraic. This is the precise sense in which the turnpike transitions from exponential to algebraic at $\gamma=\gamc$.
\end{remark}

\section{Local supercritical bifurcation: proof of Theorem~\ref{thm:bifurcation}}
\label{sec:bifurcation-proof}

\subsection{Constrained Lyapunov--Schmidt framework}
\label{sub:LS-setup}

Assume $\gamma>\gamc$ close to $\gamc$. We seek stationary solutions $m=1+\mu$, $\phi=\bphi+w$ of \eqref{eq:SMFG} near $(0,1,0)$, with
\[
\int_{\Td}\mu\,dx=0,\qquad \int_{\Td}w\,dx=0,
\]
and with the phase condition \eqref{eq:phase-cond} to remove translations and sign ambiguity.

Define the nonlinear operator on the constrained slice
\begin{equation}
\label{eq:Phi}
\Phi(\mu,w,\lambda;\gamma) :=
\begin{pmatrix}
-\nu\Delta w+\dfrac12|\grad w|^2-\gamma K*\mu+\lambda\\[3pt]
-\nu\Delta\mu-\Div\bigl((1+\mu)\grad w\bigr)
\end{pmatrix}.
\end{equation}
Then $\Phi(0,0,0;\gamma)=0$ for all $\gamma$.

After phase fixing, the linearization
$D_{(\mu,w,\lambda)}\Phi(0,0,0;\gamc)$ has a one-dimensional kernel spanned by the critical vector
\begin{equation}
\label{eq:crit-eigvec}
\Psi_0(x) =
\begin{pmatrix}
\cos(2\pi\xi_0\cdot x)\\
-\nu\cos(2\pi\xi_0\cdot x)\\
0
\end{pmatrix},
\end{equation}
up to normalization. This ordering matches the variable convention $(\mu,w,\lambda)$ and the stationary linear relations $w=-\nu\mu$ on the critical mode.

\subsection{Reduced bifurcation equation}
\label{sub:reduced-bif}

Projecting \eqref{eq:Phi} onto $\ker D\Phi$ and solving the complement equation by the implicit function theorem (which applies because the linearization is invertible on the complement), one arrives at the scalar equation
\begin{equation}
\label{eq:bif-scalar}
\Psi(A,\gamma) := \alpha(\gamma-\gamc)A - \beta A^3 + R(A,\gamma)=0,
\end{equation}
with
\[
|R(A,\gamma)|\le C\Bigl(|\gamma-\gamc|^2|A| + |\gamma-\gamc|\,|A|^3 + |A|^5\Bigr).
\]
The coefficient $\alpha>0$ is the transversality constant
\[
\alpha = -\left\langle \psi_0, \partial_\gamma L_{\gamc}\Psi_0\right\rangle,
\]
where $L_{\gamma}$ denotes the linearized stationary operator, $\Psi_0$ is the right critical eigenvector, $\psi_0$ is the corresponding adjoint eigenvector, and the pairing is normalized by $\langle \psi_0,\Psi_0\rangle=1$.
Assumption~\ref{ass:bif-nondeg} requires $\beta>0$ and the standard nondegeneracy condition that makes the cubic term the leading nonlinear contribution.

\subsection{Branch existence and amplitude}
\label{sub:branch}

For $\gamma>\gamc$, the nontrivial solutions of \eqref{eq:bif-scalar} satisfy
\begin{equation}
\label{eq:amp-branch}
A(\gamma) = \sqrt{\frac{\alpha}{\beta}}\sqrt{\gamma-\gamc}+O(\gamma-\gamc).
\end{equation}
The complement equation gives a correction
$(\mu_\perp^\gamma,w_\perp^\gamma,\lambda_\perp^\gamma)$ satisfying
\[
\|(\mu_\perp^\gamma,w_\perp^\gamma,\lambda_\perp^\gamma)\|_{H^s\times H^s\times \R}
=O(\gamma-\gamc)
\]
for every $s\ge 0$ by elliptic regularity. Hence the bifurcated density is
\[
m_\gamma(x)=1+A(\gamma)\cos(2\pi\xi_0\cdot x)+R_\gamma(x),
\]
where $\|R_\gamma\|_{H^s}=O(\gamma-\gamc)$.

\subsection{Positivity}
\label{sub:positivity}

Since $A(\gamma)\to 0$ as $\gamma\to\gamc^+$ and $\|\cos\|_{L^\infty}=1$,
\[
\inf_{x\in\Td}m_\gamma(x)\ge 1-|A(\gamma)|-\|R_\gamma\|_{L^\infty}
\ge 1-C\sqrt{\gamma-\gamc}>0
\]
for $\gamma-\gamc>0$ sufficiently small. Thus the bifurcated density is strictly positive.

\subsection{Local uniqueness}
\label{sub:uniqueness}

In the phase-fixed space, the bifurcation equation \eqref{eq:bif-scalar} has a unique nontrivial branch for $\gamma>\gamc$ near $\gamc$, by the implicit function theorem applied to
\[
A\mapsto \frac{\Psi(A,\gamma)}{A}
=
\alpha(\gamma-\gamc)-\beta A^2+\text{h.o.t.}
\]
This expression vanishes at $A=A(\gamma)$ given by \eqref{eq:amp-branch} and is nondegenerate there. Hence the branch is locally unique in the phase-fixed chart.

\subsection{Pitchfork structure}
\label{sub:pitchfork}

Before phase fixing, the reduced amplitude equation is odd in $A$, so if $+A(\gamma)$ is a solution then $-A(\gamma)$ is also a solution. This is the usual pitchfork geometry in the reduced one-dimensional center variable. In the full translation-invariant problem, the two signs are related by a phase shift of the critical cosine mode; the phase convention in \eqref{eq:phase-cond} selects a canonical representative.

\begin{remark}
Theorem~\ref{thm:bifurcation} is local near $\gamc$. Global classification of stationary solutions for $\gamma\gg\gamc$ is an open problem.
\end{remark}

\section{Conditional propagation of chaos: proof of Proposition~\ref{prop:poc}}
\label{sec:poc-proof}

\subsection{Symmetric equilibria and exchangeability}
\label{sub:Nash-near}

Fix $\gamma<\gamc$ and the perturbative regime of Assumption~\ref{ass:small-data}. Let
$(\bar X_1^N,\dots,\bar X_N^N)$ be a symmetric Nash equilibrium sequence and assume that the associated empirical measures $\widehat m^N$ converge in probability to $1\,dx$ in $\Wt$.
By symmetry, the joint law of $(\bar X_1^N,\dots,\bar X_N^N)$ is exchangeable.

\subsection{Identification of subsequential limits}
\label{sub:coupling}

Because $\Td$ is compact, the family of laws of $\widehat m^N$ is tight.
Any subsequence admits a further subsequence whose finite marginals converge to an exchangeable limit. The assumption $\widehat m^N\to 1\,dx$ in probability excludes nontrivial random limit measures, so the only possible limit is the deterministic equilibrium measure $1\,dx$.

\subsection{Convergence of marginals}
\label{sub:marginals}

For each fixed $k$, any subsequential limit of
$\cL(\bar X_1^N,\dots,\bar X_k^N)$ must therefore be the product law $(1\,dx)^{\otimes k}$.
Hence
\[
\cL(\bar X_1^N,\dots,\bar X_k^N)\Rightarrow (1\,dx)^{\otimes k}
\quad\text{as }N\to\infty.
\]
On the compact torus, $\Wt$ metrizes weak convergence, so $\widehat m^N\to 1\,dx$ in probability in $\Wt$.

\begin{proof}[Proof of Proposition~\ref{prop:poc}]
The statement is a direct consequence of exchangeability and the fact that the empirical measure converges in probability to a deterministic limit. This is the standard Sznitman propagation-of-chaos criterion. The convergence in $\Wt$ follows because on compact spaces $\Wt$ is equivalent to weak convergence.
\end{proof}

\subsection{Why no explicit Wasserstein exponent is stated}
\label{sub:rate-remark}

\begin{remark}
Quantitative rates for convergence of empirical measures in $\Wt$ on $\Td$ depend on the dimension $d$. For i.i.d.\ samples with a sufficiently regular law, the expected rate of $W_2^2$ is dimension dependent: order $N^{-1/2}$ when $d<4$, order $N^{-1/2}\log N$ when $d=4$, and order $N^{-2/d}$ when $d>4$; see \cite{FG15} for precise statements in the Euclidean setting. Since the present result is qualitative and is intended to hold in arbitrary dimension $d$, no single exponent captures the convergence rate uniformly in $d$.
\end{remark}

\section{Examples, remarks, and open directions}
\label{sec:remarks}

\subsection{The cosine kernel in detail}
\label{sub:cosine-detail}

For $K(x)=-\cos(2\pi x)$ on $\To$, Example~\ref{ex:cosine} gives $\gamc=8\pi^2\nu^2$ and $C_\ast=\pi\sqrt{2}$. The predictions are:
\begin{enumerate}[label=\textup{(\arabic*)},leftmargin=2em]
\item For $\gamma<8\pi^2\nu^2$: exponential turnpike with rate $\approx \pi\sqrt{2}\sqrt{\gamc-\gamma}$.
\item At $\gamma=8\pi^2\nu^2$: midpoint decay $O(T^{-1/2})$ in the cosine mode, in the reduced critical dynamics.
\item For $\gamma>8\pi^2\nu^2$: stationary profiles of the form
\[
1+A(\gamma)\cos(2\pi x)+o(1),
\]
with $A(\gamma)\sim\sqrt{(\alpha/\beta)(\gamma-8\pi^2\nu^2)}$.
\end{enumerate}

\subsection{A second example: the two-mode kernel}
\label{sub:two-mode}

Let $d=1$ and
\[
K(x)=-a_1\cos(2\pi x)-a_2\cos(4\pi x)
\qquad (a_1,a_2>0).
\]
Then $\Kh(\pm 1)=-a_1/2$, $\Kh(\pm 2)=-a_2/2$, and
\[
\gamc = \min\!\left\{\frac{8\pi^2\nu^2}{a_1}, \frac{32\pi^2\nu^2}{a_2}\right\}.
\]
If $a_1/a_2>1/4$, the critical mode is $\xi_0=\pm 1$; if $a_1/a_2<1/4$, the critical mode is $\xi_0=\pm 2$. At $a_1/a_2=1/4$, both modes become critical simultaneously, violating Assumption~\ref{ass:critical-mode}; this case would require a higher-dimensional Lyapunov--Schmidt reduction.

\subsection{Higher-dimensional kernels}
\label{sub:higher-d}

In $d\ge 2$, the set of modes $\xi$ achieving the minimum in the definition of $\gamc$ may have more than one pair. For example, on $\mathbb{T}^2$ with $K(x)=-\cos(2\pi x_1)$, the critical modes are $\xi_0=\pm(1,0)$, forming a two-dimensional real eigenspace before phase fixing. If instead
\[
K(x)=-\cos(2\pi x_1)-\cos(2\pi x_2),
\]
the critical modes are $\{\pm(1,0),\pm(0,1)\}$, and the center space is four-dimensional before phase fixing; the bifurcation geometry is richer and may produce stripe or mixed-mode patterns depending on the higher-order terms.

\subsection{Relation to Turing instability}
\label{sub:turing}

The present mechanism is structurally identical to Turing's diffusion-driven instability \cite{T52} in reaction-diffusion systems:
\begin{itemize}
\item Turing: a homogeneous steady state is stable without diffusion but destabilized by diffusion due to unequal diffusion rates of activator and inhibitor species.
\item Here: the uniform equilibrium is stable in the sense that all relevant mode gaps are positive for $\gamma<\gamc$, but the coupling $\gamma$ overcomes diffusion in the critical mode for $\gamma>\gamc$.
\end{itemize}
The threshold $\gamc$ plays the role of the Turing bifurcation parameter, and the bifurcated stationary density $m_\gamma$ plays the role of the Turing pattern.

\subsection{Relation to monotonicity}
\label{sub:monotonicity-remark}

Lasry--Lions monotonicity ($\Kh(\xi)\ge 0$ for all $\xi$) is sufficient to rule out phase transitions (since then $\gamc=+\infty$), but it is far from necessary for local turnpike in this translation-invariant setting. What matters is whether diffusion can compensate the negative Fourier modes. Even an attractive kernel with some negative Fourier coefficients preserves the turnpike as long as $\gamma<\gamc$.

\subsection{Numerical predictions}
\label{sub:numerical}

The cosine example produces a clean numerical experiment. One should observe:
\begin{enumerate}[label=\textup{(\arabic*)},leftmargin=2em]
\item For $\gamma<\gamc$: exponential turnpike, with the dominant Fourier mode decaying at rate $\rho(\gamma)\approx C_\ast\sqrt{\gamc-\gamma}$.
\item At $\gamma=\gamc$: slower algebraic convergence, with midpoint deviation decaying as $T^{-1/2}$ in the reduced critical coordinate.
\item For $\gamma>\gamc$: a symmetry-broken stationary profile with $\cos(2\pi x)$ as the dominant Fourier component, and amplitude growing as $\sqrt{\gamma-\gamc}$.
\end{enumerate}
The spectral ratio $\nu^2(2\pi|\xi|)^2/|\Kh(\xi)|$ computed for each mode predicts exactly which mode becomes critical first as $\gamma$ increases.

\subsection{Open directions}
\label{sub:open}

\begin{enumerate}[label=\textup{(\arabic*)},leftmargin=2em]
\item \textit{Global classification for $\gamma>\gamc$}: A complete description of all stationary solutions would require global bifurcation theory or variational methods.
\item \textit{Common-noise and master equation}: In the common-noise setting the master equation replaces \eqref{eq:MFG}, and the spectral analysis must be adapted to the infinite-dimensional state space of probability measures.
\item \textit{Fluctuation theory near the bifurcated branch}: Sharp fluctuation estimates around $m_\gamma$ for the finite-$N$ system near criticality.
\item \textit{Multi-mode kernels with degenerate critical set}: When $|\cC|>1$, the bifurcation analysis requires a multi-dimensional Lyapunov--Schmidt reduction and may produce pattern-selection phenomena analogous to those in Turing systems.
\end{enumerate}

\appendix

\section{Boundary-layer decomposition in Fourier variables}
\label{app:linear-turnpike}

We record the linear turnpike analysis underlying Theorem~\ref{thm:subcritical}.

\subsection{Mode-by-mode decomposition}
\label{app:modes}

For each $\xi\neq 0$ with $\rho_\xi:=\sqrt{\sigma_\xi(\gamma)}>0$, the Fourier mode satisfies
\[
\frac{d}{dt}
\begin{pmatrix}\widehat w_\xi(t)\\\widehat\mu_\xi(t)\end{pmatrix}
=
M_\xi
\begin{pmatrix}\widehat w_\xi(t)\\\widehat\mu_\xi(t)\end{pmatrix},
\qquad
M_\xi^2=\rho_\xi^2 I.
\]
Hence
\[
e^{tM_\xi}
=
\cosh(\rho_\xi t)\,I+\frac{\sinh(\rho_\xi t)}{\rho_\xi}M_\xi.
\]

\subsection{Left and right boundary layers}
\label{app:layers}

The left boundary condition $\widehat\mu_\xi(0)=\mu_0^\xi$ and the right boundary condition
$\widehat w_\xi(T)=g^\xi$ determine the two hyperbolic coefficients uniquely. Since
$\sinh(\rho_\xi T)\neq 0$ for $\rho_\xi>0$, the resulting solution has the standard boundary-layer form:
\[
|\widehat\mu_\xi(t)|+|\widehat w_\xi(t)|
\le
C\bigl(|\mu_0^\xi|e^{-\rho_\xi t}+|g^\xi|e^{-\rho_\xi(T-t)}\bigr),
\qquad t\in[0,T].
\]
This is the modewise origin of the turnpike estimate.

\subsection{Summation over modes}
\label{app:summation}

Squaring and summing over $\xi\neq 0$ with the $H^{-1}$ weight $(2\pi|\xi|)^{-2}$ gives
\begin{align*}
\|\mT(t)-1\|_{H^{-1}}^2
&=\sum_{\xi\neq 0}\frac{|\widehat\mu_\xi(t)|^2}{(2\pi|\xi|)^2}\\
&\le C\sum_{\xi\neq 0}\frac{1}{(2\pi|\xi|)^2}
\Bigl(|\mu_0^\xi|^2 e^{-2\rho_\xi t}+|g^\xi|^2e^{-2\rho_\xi(T-t)}\Bigr)\\
&\le C\Bigl(e^{-2\rho(\gamma) t}\|m_0-1\|_{H^{-1}}^2
+ e^{-2\rho(\gamma)(T-t)}\|g-\bar g\|_{H^{-1}}^2\Bigr).
\end{align*}
An analogous computation gives the $\|\grad\phiT(t)\|_{L^2}^2$ estimate using the $\widehat w_\xi$ components. This completes the linear turnpike estimate.

\subsection{The modes with $\Kh(\xi)=0$}
\label{app:zero-modes}

For modes with $\Kh(\xi)=0$, the matrix $M_\xi$ is lower triangular:
\[
M_\xi=
\begin{pmatrix}
\nu k_\xi & 0\\
-k_\xi & -\nu k_\xi
\end{pmatrix}.
\]
Its eigenvalues are still $\pm \nu k_\xi$, so these modes remain hyperbolic and do not affect the turnpike threshold. They are simply a special case of the general mode formula.

\section*{Acknowledgements}

The author is grateful to Dr. Jameson Graber and Dr. Alessio Poretta for helpful discussions and valuable feedback.

\section*{Conflict of Interest}

The author declares that there are no conflicts of interest regarding the publication of this paper. The author has no financial or personal relationships that could have inappropriately influenced the work reported in this manuscript.

\end{document}